\documentclass[a4paper,11pt]{article}

\usepackage{latexsym,graphics,color}
\usepackage{amsmath,amsfonts,amssymb,amsthm}
\newtheorem{thm}{Theorem}[section]
\newtheorem{lem}[thm]{Lemma}

\newtheorem{cor}[thm]{Corollary}
\newtheorem{prop}[thm]{Proposition}

\newtheorem{remark}[thm]{Remark}

\input epsf

\title{Euclid meets Popeye: The Euclidean Algorithm for $2\times 2$ Matrices}
\author{Roland Bacher}

\begin{document}
\maketitle

\begin{abstract}\footnote{
 Math. class: Primary:
 11A05.
 Secondary: 11H06, 11J70.
} An analogue of the Euclidean algorithm for square matrices of
size $2$ with integral non-negative entries and strictly positive
determinant $n$ defines a finite set $\mathcal R(n)$ of Euclid-reduced
matrices corresponding to elements of
$\{(a,b,c,d)\in\mathbb N^4\ \vert\ n=ab-cd,\ 0\leq c,d<a,b\}$.
With Popeye's help\footnote{Acknowledged by his appearance in the title (he refused co-authorship on the pretext of a weak contribution due to a poor spinach-harvest).}  on the use of sails of lattices
we show that $\mathcal R(n)$ contains
$\sum_{d\vert n,\ d^2\geq n}\left(d+1-\frac{n}{d}\right)$ elements.
\end{abstract}



\section{Introduction}

We denote by $\mathbb N=\{0,1,2,\ldots\}$ the set of all non-negative integers
and by $\mathcal P=\{\left(\begin{array}{cc}a&b\\c&d\end{array}\right)\ \vert a,b,c,d\in \mathbb N,\ ad-bc>0\}$ the set of all square matrices of size $2$
with entries in $\mathbb N$ and strictly positive determinant.
The subset of matrices of determinant $n$ in
$\mathcal P$ is written as $\mathcal P(n)$.

An \emph{elementary reduction} of a matrix $M$ is a matrix
in $\{EM,E^tM,ME,ME^t\}$ where
$E=\left(\begin{array}{cc} 1&-1\\0&1 \end{array}\right)$.
Elementary reductions of $M$ subtract a row/column
from the other row/column of $M$.

A matrix $M$ in $\mathcal P$ is
\emph{Euclid-reduced} if and only if $\mathcal P$ contains
no elementary reduction of $M$. Equivalently,
$M=\left(\begin{array}{cc} a&b\\c&d\end{array}\right)$
in $\mathcal P$
is Euclid-reduced if $\min(a,d)>\max(b,c)$.

We denote by $\mathcal R$ the subset of Euclid-reduced matrices in $\mathcal P$
and by $\mathcal R(n)=\mathcal R\cap\mathcal P(n)$ the subset of $\mathcal R$
corresponding to Euclid-reduced matrices of determinant $n$.

The main result of this paper describes the number
$\sharp(\mathcal R(n))$ of elements
in the set $\mathcal R(n)$ of Euclid-reduced matrices of determinant $n$:

  \begin{thm}\label{thmmain}
  The number of elements $(a,b,c,d)$ in $\mathbb N^4$
  such that $n=ab-cd$ and $\min(a,b)>\max(c,d)$ is given by  
\begin{align}\label{sumEuclirr}
&  \sum_{d\vert n,\ d^2\geq n}\left(d+1-\frac{n}{d}\right)\ .\\
\end{align}
  The map $(a,b,c,d)\longmapsto
  \left(\begin{array}{cc}a&c\\d&b
        \end{array}\right)$ is a one-to-one correspondence
      between such solutions and elements 
      in the set $\mathcal R(n)$ of Euclid-reduced matrices having
      determinant $n$.
  \end{thm}

  All summands occuring in (\ref{sumEuclirr}) are strictly positive and
  the last summand (corresponding to the trivial divisor $d=n$ of $n$)
  equals $n$.
  We have therefore $\sharp(\mathcal R(n))\geq n$ with  equality
  for $n>1$ if and only if $n$ is a prime number.
  Our proof of Theorem \ref{thmmain} shows that
  solutions associated to a prime number $p$
  are in one-to-one correspondence with the
  $p$ sublattices of index $p$ in $\mathbb Z^2$
  which do not contain the vector $(1,1)$.

  Similarly, $\sharp(\mathcal R(n))=n+1$
  if and only if $n=p^2$ is the square of prime number $p$. 

  Cardinalities of the sets $\mathcal R(1),\mathcal R(2),\ldots$
  are given by the integer sequence
  $$1,2,3,5,5,8,7,11,10,14,11,19,13,20,18,24,17,30,19,31,\ldots$$
  not yet recognized by The Online-Encyclopedia of Integer Sequences
  \cite{OEIS}.
  
  Klein's Vierergruppe $\mathbb V$ (underlying the $2$-dimensional vector
  space over the field of two elements) acts on solutions $(a,b,c,d)$
  by permuting the first two entries, the last two entries or
  the first two and the last two entries. We denote by
  $\mathcal O=\{(a,b,c,d),(b,a,c,d),(a,b,d,c),(b,a,d,c)\}$
  the orbit of a solution $(a,b,c,d)$ under the action of $\mathbb V$.
  The following lists give lexicographically largest representants of all
  orbits for the sets of solutions associated to the prime
  numbers $11,13$ and $17$:
$$\begin{array}{cccc|c}
    a&b&c&d&\sharp(\mathcal O)\\
    \hline
    11&1&0&0&2\\
    6&2&1&1&2\\
    4&3&1&1&2\\
    5&3&2&2&2\\
    5&4&3&3&2\\
    6&6&5&5&1\\
    \hline
     &&&&11
  \end{array}\qquad
\begin{array}{cccc|c}
    a&b&c&d&\sharp(\mathcal O)\\
    \hline
    13&1&0&0&2\\
    7&2&1&1&2\\
    5&3&2&1&4\\
  4&4&3&1&2\\
  5&5&4&3&2\\
    7&7&6&6&1\\
    \hline
     &&&&13
  \end{array}\qquad
\begin{array}{cccc|c}
    a&b&c&d&\sharp(\mathcal O)\\
    \hline
    17&1&0&0&2\\
    9&2&1&1&2\\
    6&3&1&1&2\\
  5&4&3&1&4\\
  7&3&2&2&2\\
  5&5&4&2&2\\
  7&6&5&5&2\\
    9&9&8&8&1\\
    \hline
     &&&&17
  \end{array}
  $$

  For $n=12,14,15$ we get
  \begin{align*}
    \sharp(\mathcal S_{12})&=(4+1-3)+(6+1-2)+(12+1-1)=19,\\
    \sharp(\mathcal S_{14})&=(7+1-2)+(14+1-1)=20,\\
    \sharp(\mathcal S_{15})&=(5+1-3)+(15+1-1)=18.
  \end{align*}
  The associated lexicographically largest solutions in orbits are given by
$$\begin{array}{cccc|c}
    a&b&c&d&\sharp(\mathcal O)\\
    \hline
    12&1&0&0&2\\
    6&2&0&0&2\\
    6&2&1&0&4\\
    4&3&0&0&2\\
    4&3&1&0&4\\
    4&3&2&0&4\\
    4&4&2&2&1\\
    \hline
     &&&&19
  \end{array}\qquad
\begin{array}{cccc|c}
    a&b&c&d&\sharp(\mathcal O)\\
    \hline
    14&1&0&0&2\\
  7&2&0&0&2\\
  7&2&1&0&4\\
    5&3&1&1&2\\
  4&4&2&1&2\\
  6&3&2&2&2\\
  5&4&3&2&4\\
  6&5&4&4&2\\
    \hline
     &&&&20
  \end{array}\qquad
\begin{array}{cccc|c}
    a&b&c&d&\sharp(\mathcal O)\\
    \hline
    15&1&0&0&2\\
  5&3&0&0&2\\
  5&3&1&0&4\\
    5&3&2&0&4\\
  8&2&1&1&2\\
  4&4&1&1&1\\
  6&4&3&3&2\\
  8&8&7&7&1\\
    \hline
     &&&&18
  \end{array}
  $$

  It is perhaps worthwhile to note that non-negative
  integral solutions of $n=ab+cd$ with 
  $\min(a,b)>\max(c,d)$ are also interesting:
  For $n=p$ an odd prime there are $(p+1)/2$ solutions.
  If $p$ is congruent to $1$ modulo $4$, the number $(p+1)/2$ of
  such solutions is odd and the action of Klein's Vierergruppe
  has a fixed point expressing $p$ as a sum of two squares, see \cite{Baquix}.
  
  The sequel of this paper is organized as follows:

  Section \ref{sectcoprime} uses Moebius inversion in order
  to obtain the number of elements
  with coprime entries in $\mathcal R(n)$.

  Section \ref{sectlatt} recalls a well-known formula
  for the number of sublattices of index $n$
  in $\mathbb Z^2$. We give an elementary proof.

  Unless stated otherwise, a lattice is always a discrete subgroup
  isomorphic to $\mathbb Z^2$ of the Cartesian coordinate plane
  $\mathbb R^2$ considered as a vector space.

  Section \ref{sectsail} describes the sail of a lattice $\Lambda$
  contained in the Cartesian coordinate plane $\mathbb R^2$.

  Section \ref{sectproof} is devoted to the proof of Theorem \ref{thmmain}.

  Section \ref{sectcompl} contains a few complements: An elementary proof
  for finiteness of the set $\mathcal R(n)$, a short discussion on
  matrices of larger size or of determinant $0$. It ends with the description
  of a perhaps interesting variation over the ring of Gau\ss ian integers.

  
\section{Coprime solutions}\label{sectcoprime}

Let $\mathcal R'(n)$ denote
the subset of $\mathcal R(n)$ containing all Euclid-reduced matrices
with coprime entries. Dividing all entries of matrices in $\mathcal R(n)$ by
their greatest common divisor, we get a bijection
between $\mathcal R(n)$ and $\cup_{d,d^2\vert n}\mathcal R'(n/d^2)$
showing the
identity $\sharp(\mathcal R(n))=\sum_{d,\ d^2\vert n}
\sharp(\mathcal R'(n/d^2))$.
Moebius inversion of this identity yields now the formula 
\begin{align}\label{moebiusinv}
  \sharp(\mathcal R'(n))&=\sum_{d^2\vert n}\mu(d)\sharp(\mathcal R(n/d^2))
\end{align}
(where the Moebius function $\mu$ is defined by $\mu(n)=(-1)^e$ if
$n$ is a product of $e$ distinct primes and $\mu(n)=0$ if $n$ has a non-trivial
square-divisor).

Observe that $\mathcal R'(n)=\mathcal R(n)$ if and only if $\mu(n)\not=0$.

Cardinalities of $\mathcal R'(1),\mathcal R'(2),\ldots$ yield the
integer sequence
$$1,2,3,4,5,8,7,9,9,14,11,16,13,20,18,19,17,28,19,26,\ldots$$
not yet contained in \cite{OEIS}.

\begin{remark}
Formula (\ref{moebiusinv}) is the analogue of the identity
$$\phi(n)=\sum_{d\vert n}\mu(d)n/d$$
(where $\phi(n)=\sum_{d\vert n}\mu(d)n/d$ is the M\"obius inversion of
the trivial identity $n=\sum_{d\vert n}\phi(n)$)
for Euler's totient function
$\phi(n)=n\prod_{p\vert n}\left(1-\frac{1}{p}\right)$
counting the number of invertible classes
in $\mathbb Z/n\mathbb Z$.
\end{remark}

\section{Sublattices of finite index in $\mathbb Z^2$}\label{sectlatt}

The following well-known result (see Remark \ref{remlattgen}
below) is a crucial ingredient
for proving Theorem \ref{thmmain}. We give an elementary proof for the comfort
of the reader.

\begin{thm}\label{thmsublatZ2} The lattice $\mathbb Z^2$ has $\sum_{d,d\vert n} d$ different sublattices of index $n$.
\end{thm}

\begin{proof} Let $\Lambda$ be a sublattice of index $n$ in $\mathbb Z^2$.
  The order $d$ of $(1,0)$ in the finite quotient group $\mathbb Z^2/\Lambda$
  is therefore a divisor of $n$ and we have $\Lambda\cap \mathbb Z(1,0)=\mathbb Z(d,0)$. There exists therefore a unique element $a$ $\{0,\ldots,d-1\}$
  such that $\Lambda=\mathbb Z(d,0)+\mathbb Z(a,n/d)$. This shows that
  the lattice
  $\mathbb Z^2$ has $d$ different sublattices of index $n$
  intersecting $\mathbb Z(1,0)$ in $\mathbb Z(d,0)$ for every divisor $d$ of
  $n$. Summing over all divisors yields the result.
\end{proof}

\begin{remark}\label{remlattgen} More generally,
the number of 
enumerates sublattices of index $n$ in $\mathbb Z^d$ is given by 
\begin{align}\label{formexactsigma}
&\prod_{p\vert n}\left(\begin{array}{cc}e_p+d-1\\d-1
\end{array}\right)_p
\end{align}
(see e.g. \cite{Gr} or \cite{Zou})
where 
$\prod_{p\vert n}p^{e_p}=n$ is the factorization of $n$ into prime-powers
and where
$$\left(\begin{array}{cc}e_p+d-1\\d-1
\end{array}\right)_p=\prod_{j=1}^{d-1}\frac{p^{e_p+j}-1}{p^j-1}$$
is the evaluation of the $q$-binomial
$$\left[\begin{array}{cc}e_p+d-1\\d-1
\end{array}\right]_q=\frac{[e_p+d-1]_q!}{[e_p]_q!\ [d-1]_q!}$$
(with $[k]_q!=\prod_{j=1}^k\frac{q^j-1}{q-1}$) at the prime-divisor $p$
of $n$.

Formula (\ref{formexactsigma}) boils of course down to $\sum_{d,d\vert n}d$
if $d=2$.
\end{remark}

\section{The sail of a lattice}\label{sectsail}

Sails of lattices in $\mathbb R^d$, introduced and studied by V. Arnold,
cf. e.g\cite{Ar},  are a possible generalization of continued
fraction expansions to higher dimension. We define and discuss here
only the case $d=2$ corresponding to ordinary continued fractions.


We denote by 
$Q_{\mathrm{I}}=\{(x,y)\ \vert 0\leq x,y\}$ the closed first quadrant  containing all points with non-negative coordinates of
the Cartesian coordinate plane $\mathbb R^2$.

The \emph{sail}
$\mathcal S=\mathcal S(\Lambda)$ of a lattice $\Lambda\subset \mathbb R^2$
is the boundary with respect to the closed first quadrant $Q_{\mathrm{I}}$
of the convex hull of all non-zero elements
$(\Lambda\setminus (0,0))\cap Q_{\mathrm{I}}$ of $\Lambda$
contained in $Q_{\mathrm I}$.

The sail $\mathcal S$ of a lattice $\Lambda$
is a piecewise linear path with 
vertices in $\Lambda$ which intersects every
$1$-dimensional subspace of finite strictly positive slope in a
unique point.
Affine pieces of sails have finite strictly negative slopes.
Any affine line intersecting a sail in two
points has therefore finite strictly negative slope.

Each coordinate axis intersects a sail either in a unique
point
(this happens if and only if the coordinate axis contains
infinitely many points of the underlying lattice $\Lambda$)
or is an asymptote of the sail (if $\Lambda$ contains no
non-zero elements of the
coordinate axis).

The sail $\mathcal S(\Lambda)$ of a sublattice $\Lambda$ of index $n$ in
$\mathbb Z^2$ is always bounded with endpoints
$(\alpha_x,0),(0,\omega_y)$ for two divisors $\alpha_x$ and $\omega_y$
of $n$ such that $\alpha_x\omega_y\geq n$.

Two distinct lattice elements $u,v\in \Lambda$ on the sail $S=
\mathcal S(\Lambda)$ of a lattice $\Lambda$
are \emph{consecutive} if
the open segment joining $u$ and $v$
is contained in $\mathcal S\setminus \Lambda$.

\begin{lem}\label{lemsailbasis} Two distinct lattice elements $u,v$ on the
  sail $\mathcal S(\Lambda)\cap \Lambda$ of a lattice $\Lambda$
  generate $\Lambda$ if and only if they are consecutive.
\end{lem}

\begin{proof} Since all non-zero lattice points in $\mathcal Q_{\mathrm{I}}$
  belong to the unbounded
  convex region of $\mathcal Q_{\mathrm{I}}\setminus \mathcal S$,
  the closed triangle $\Delta=\Delta(u,v)$ with vertices $(0,0),u,v$ 
  contains no other element of $\Lambda$ if and only if $u$ and $v$
  are consecutive.

  Pairs of consecutive points $u,v$ generate $\Lambda$
  since $\Delta\cup(-\Delta)$ is a fundamental domain for
  the lattice spanned by $u$ and $v$.
\end{proof}

A \emph{sailbasis} of a lattice $\Lambda$ is a basis of $
\Lambda$ consisting of two consecutive elements in 
the sail $\mathcal S$ of $\Lambda$. Every lattice has a sailbasis.

Two linearly independent elements $u,v$ in the the first quadrant
$Q_{\mathrm I}$ form a sailbasis of the
lattice $\mathbb Z u+\mathbb Z v$ generated by $u$ and $v$ 
if and only if the affine line containing $u$ and $v$
has finite strictly negative slope.

\begin{remark} Sails are generalisations of
  continued fractions: Given a real number
  $\theta$, vertices of the sail for the lattice $e^{-i\arctan(\theta)}
(\mathbb Z+i\mathbb Z)\cap [0,\infty]+i[0,\infty]\setminus\{0\}$
correspond essentially to convergents of $\theta$, see for example \cite{Ar}.
\end{remark}

\section{Proof of Theorem \ref{thmmain}}\label{sectproof}

A sailbasis $u,v$ of a lattice is \emph{central} if the open segment
joining $u$ and $v$ intersects the
diagonal line $x=y$. The two elements of a central sailbasis belong therefore
to different connected components of $\mathbb R^2\setminus \mathbb R(1,1)$.
Every lattice has at most a unique central sailbasis.

A lattice $\Lambda$ is \emph{bad} if it has no central sailbasis.
Equivalently, a lattice is bad if its sail $\mathcal S$ intersects the set
$\Lambda\cap\mathbb R(1,1)$ of diagonal lattice-elements.

A sailbasis $u,v$ of a bad lattice $\Lambda=\mathbb Z u+\mathbb Z v$
is \emph{normalized} if $u$ in $\mathbb R (1,1)$ is a diagonal element
and $v$ belongs to
the open halfplan $\{(x,y)\ \vert\ x>y\}$ below the diagonal line. 
Lemma \ref{lemsailbasis}
shows that a bad lattice $\Lambda$ has a unique normalized
sailbasis given by $u=\mathcal S\cap\mathbb R(1,1)$ and by the unique
consecutive element $v$ in $\mathcal S\cap \Lambda$ of $u$ which lies below
the diagonal line $x=y$.

\begin{prop}\label{propbadlatt} The lattice $\mathbb Z^2$ contains
  $$\sum_{d,\ d^2<n,\ d\vert n}d+\sum_{d,\ d^2\geq n,\ d\vert n}(n/d-1)$$
  bad sublattices of index $n$.
\end{prop}

\begin{proof} Bad lattices are in one-to-one correspondence
  with their normalized sailbases. We count them by adapting
  the proof of Theorem \ref{thmsublatZ2}.
  
  Let $u=(d,d)$ in $\Lambda\cap\mathcal S$ be the diagonal element 
  of a normalized sailbasis $u,v$ generating a
  bad sublattice $\Lambda=\mathbb Zu+\mathbb Z v$
  of index $n$ in $\mathbb Z^2$.
  The image of the element $(1,1)$ in the quotient group $\mathbb Z^2/\Lambda$
  is therefore of order $d$ dividing $n$. Since $u,v$ is
  a sailbasis, the coefficients $v_x,v_y$ of the 
  remaining basis element
  $v=(v_x,v_y)$ satisfy the inequalities $0\leq v_y<d<v_x$.
  Since $\Lambda=\mathbb Z u+\mathbb Z v$ is a sublattice of
  index $n$ in $\mathbb Z^2$, the element $v$ of $\mathbb N^2$
  belongs to the line $(n/d,0)+\mathbb R(1,1)$.
  We have therefore  $v=(n/d+a,a)$ for a suitable non-negative integer $a$.

  If $d<\sqrt{n}$, the trivial inequalities $d<n/d\leq n/d+a=v_x$
  imply $v_x>d$ for all choices of $a$ in $\mathbb N$.
  The inequality $v_y<d$ implies that $a=v_y$ belongs to
  the set $\{0,1,2,\ldots,d-1\}$ of the $d$ smallest non-negative
  integers. For every divisor $d<\sqrt{n}$ there are
  therefore $d$ bad sublattices of index $n$
  containing $(d,d)$ in their sail.

  If $d$ is a divisor of $n$ such that $d\geq \sqrt{n}$ the
  inequality $d<v_x=n/d+a$ implies $a\geq d-n/d+1\geq 0$.
  We have also $a=v_y<d$ leading to $a$ in the set
  $\{d-n/d+1,d-n/d+2,\ldots,d-1\}$ containing $n/d-1$ elements.

  Summing over all contributions given by divisors of $n$
  ends the proof.
\end{proof}

\begin{proof}[Proof of Theorem \ref{thmmain}]
  Solutions of $ab-cd=n$ with $\min(a,b)>\max(c,d)$ are in one-to-one
  correspondence with central sailbases $(a,d),(d,b)$ generating sublattices
  of index $n$ in $\mathbb Z^2$.
  The number of elements in $\mathcal R(n)$ is therefore obtained by
  subtracting the number $\sum_{d,\ d^2<n,\ d\vert n}d+\sum_{d,\ d^2\geq n,\ d\vert n}(n/d-1)$ of bad lattices
of index $n$ in $\mathbb Z^2$ given by Proposition \ref{propbadlatt}
from the total number $\sum_{d,d\vert n}d$
of lattices of index $n$ in $\mathbb Z^2$
given by Theorem \ref{thmsublatZ2}. Simplification yields the result.
\end{proof}

\section{Complements}\label{sectcompl}

  \subsection{Finiteness}

  We discuss in this Section a few finiteness properties of Euclid-reducedness.

  First, we give an elementary proof of finiteness for the number
  of Euclid-reduced matrices in $\mathcal P$ of given determinant
  which does not make use of Theorem \ref{thmmain}.

  We consider then briefly the case of matrices of size larger than
  $2$ and of square matrices of size two with determinant $0$.

\subsection{An easy bound on entries of Euclid-reduced matrices}  

\begin{prop}\label{propeasybound} Matrices in $\mathcal R(n)$
  involve only entries in $\{0,1,\ldots,n\}$.
\end{prop}

\begin{cor}\label{corboundedentries} There are at most
  $(n+1)^4$ matrices in the set $\mathcal R(n)$ of Euclid-reduced matrices
  of determinant $n$.
\end{cor}  

We leave the obvious proof of the Corollary to the reader.

\begin{proof}[Proof of Proposition \ref{propeasybound}]
  Let $n=ab-cd$ with $\min(a,b)>\max(c,d)$ be a solution
  corresponding to the Euclid-reduced matrix
  $\left(\begin{array}{cc} a&c\\d&b\end{array}\right)$
  with $\max(a,b)$ maximal among entries occuring in elements of
  $\mathcal R(1),\ldots,\mathcal R(n)$.
  Up to exchanging $a$ and $b$
  we can suppose that $a\geq b$.
  Since
  $n=ab-cd\geq ab-(b-1)^2>0$ we can assume $c=d=b-1$.
  Restricting $ax-(x-1)^2$ to $x$ in $[1,\ldots,a]$ we can furthermore
  assume either $x=1$ or $x=a$. In the first case we get $n\geq a\cdot 1-0^2=a$
  and in the second case we get $n\geq a^2-(a-1)^2=2a-1$ showing
  the inequality $\max(a,b)=a\leq n$ in both cases.
\end{proof}

\subsection{Finiteness for size larger than two}

The obvious definition of Euclid-reducedness leads to infinite sets
of Euclid-reduced matrices when considering matrices of larger size:

    The matrix $\left(\begin{array}{ccc}4+x&2+x&1+x\\x&1+x&3+x\\1+x&1+x&2+x\end{array}\right)$ has determinant $1$ and is 'Euclid-reduced' for any natural integer $x$.

    An example of size $2\times 3$ with columns generating $\mathbb Z^2$ is given by
    is given by $\left(
      \begin{array}{ccc}n&3&2\\1&2&3\end{array}\right)$ for $n\geq 5$ such that
    $n\not\equiv 4\pmod 5$.

\subsection{Finiteness for determinant zero}
All square matrices of size two with (at least) three zero entries and
an arbitrary entry in $\mathbb N$ are Euclid-reduced and every
Euclid-reduced matrix with determinant $0$ and entries in $\mathbb N$ is
of this form: If a matrix $M$ (of square size two with entries in $\mathbb N$
has determinant $0$ then its rows (or columns) are linearly dependent.
Subtracting the smaller row iteratedly from the larger one we end up with
a matrix having a zero-row. Working with columns we get finally a matrix
having a unique non-zero entry.

Requiring the entries of such a matrix to have a given non-zero greatest
divisor ensures uniqueness up to the location of the non-zero entry.
There are therefore exactly four Euclid-reduced matrices (of square size $2$)
with determinant $0$ and and greatest common divisor of entries a given non-zero
integer $d\geq 1$.

\subsection{Gau\ss ian integers}

We discuss briefly an analogue of $\mathcal R(n)$ over the ring
of Gau\ss ian integers (the case of integers in an imaginary quadratic 
number field is probably similar).

Given a non-zero Gau\ss ian integer $z$, we define the set $\mathcal S(z)$
containing all solutions
of $ab+cd=z$ satisfying $\min(\vert a\vert,\vert b\vert)>\max(\vert c\vert,\vert d\vert)$ with $a,b,c,d$ in the set $\mathbb Z[i]$ of Gau\ss ian integers.

The two identities
$$2m+1=(2n+(2n^2-m-1)i)(2n-(2n^2-m-1)i)-(2n^2-m)^2$$
and
$$2m=(2n+1+(2n^2+2n-m)i)(2n+1-(2n^2+2n-m)i)-(2n^2+2n-m+1)^2$$
show that
the sets $\mathcal S(z)$ are always infinite for $z\in\mathbb Z\setminus\{0\}$.

More generally, $\mathcal S(z)$ is infinite for every Gau\ss ian integer
$z$ of the form $z=nu^2$ for $n$ in $\mathbb N\setminus\{0\}$ a sum
of two squares (i.e. containing no odd power of a prime congruent to $3$ modulo
$4$ in its prime-factorization) and for $u\in
\mathbb Z[i]\setminus\{0\}$ an arbitrary non-zero Gau\ss ian integer.

Solutions can be fairly large as shown by the identity
$$2+3i=-(7-18)^2+(3+19i)(-15+12i)$$
contributing to $\mathcal S(2+3i)$ which has seemingly only finitely many
elements.

There are obvious bijections between $\mathcal S(z),\mathcal S(\overline{z}),
\mathcal S(-z),\mathcal S(\pm iz)$. Moreover, $\mathcal S(z)$ infinite
implies $\mathcal S(s\overline{s}t^2z)$ infinite for every 
non-zero Gau\ss ian integers $s,t$ in $\mathbb Z[i]\setminus\{0\}$.


\noindent Roland BACHER, 

\noindent Univ. Grenoble Alpes, Institut Fourier, 

\noindent F-38000 Grenoble, France.
\vskip0.5cm
\noindent e-mail: Roland.Bacher@univ-grenoble-alpes.fr

\end{document}